\documentclass{amsart}
\usepackage{fullpage}
\usepackage{amsmath,amsfonts,amssymb}
\usepackage{tikz,pstricks}
\usepackage[mathcal]{euscript}
\newtheorem{theorem}{Theorem}[section]
\newtheorem{lemma}[theorem]{Lemma}

\newtheorem{proposition}[theorem]{Proposition}
\newtheorem{defi}{Definition}[section]
\newtheorem{remark}{Remark}[section]
\newcommand{\apriori}{\emph{a priori\/}}
\newcommand{\card}{{\rm card}}
\newcommand{\ddt}{\hspace{.5ex}{\rm d}t}
\newcommand{\diam}{{\rm diam}}
\newcommand{\DIR}{{\rm D}}
\newcommand{\disc}{\mathcal{M}}
\newcommand{\discd}{\mathcal{D}}
\newcommand{\dom}{\Omega}
\newcommand{\ds}{\hspace{.5ex}{\rm d}s}
\newcommand{\dt}{\delta \hspace{-0.05ex} t}
\newcommand{\dtphi}{\partial_{t,\discd}(\varphi)}
\newcommand{\dtu}{\partial_{t,\discd}(u)}
\newcommand{\dx}{\hspace{.5ex}{\rm d}x}
\newcommand{\edge}{\sigma}
\newcommand{\edges}{\mathcal{E}}
\newcommand{\edgesext}{\mathcal{E}_{{\rm ext}}}
\newcommand{\edgesint}{\mathcal{E}_{{\rm int}}}
\newcommand{\eqdef}{:=}  %{\stackrel{\mathrm{def}}{=}}
\newcommand{\grad}{\nabla}
\newcommand{\ie}{\emph{i.e.\/}}
\newcommand{\lap}{\Delta}
\newcommand{\mesh}{\mathcal{M}}
\newcommand{\NEU}{{\rm N}}
\newcommand{\nnn}{{n\in\xN}}
\newcommand{\nti}{{n \to + \infty}}
\newcommand{\norm}[2]{\left\| #1 \right\|_{#2}}
\newcommand{\normedeu}[2]{\|#1\|_{L^2(#2)}}
\newcommand{\normLdHmu}[1]{\|#1\|_{\xLtwo(0,T;\, \xHdisc^{-1})}}
\newcommand{\normLdHu}[1]{\|#1\|_{\xLtwo(0,T;\, \xHdisc^1)}}
\newcommand{\normLds}[1]{\|#1\|_{\xLtwo(0,T;\, {\rm H}_{\disc,\ast})}}
\newcommand{\normLdms}[1]{\|#1\|_{\xLtwo(0,T;\, {\rm H}_\disc^\ast)}}
\newcommand{\normemundisc}[1]{\|#1\|_{-1,\disc}}
\newcommand{\normeundisc}[1]{\|#1\|_{1,\disc}}
\newcommand{\normes}[1]{\|#1\|_{\ast}}
\newcommand{\normems}[1]{\|#1\|^{\ast}}
\newcommand{\snormeundisc}[1]{|#1|_{1,\disc}}
\newcommand{\snormLdHu}[1]{|#1|_{\xLtwo(0,T;\, \xHdisc^1)}}
\newcommand{\xCinfty}{{\rm C}^{\infty}}
\newcommand{\xHdisc}{{\rm H}_{\disc}}
\newcommand{\xHdiscd}{{\rm H}_{\discd}}
\newcommand{\xHn}[1]{{\rm H}^#1}
\newcommand{\xHone}{{\rm H}^{1}}
\newcommand{\xLinfty}{{\rm L}^\infty}
\newcommand{\xLn}[1]{{\rm L}^#1}
\newcommand{\xLtwo}{{\rm L}^2}
\newcommand{\xN}{{\mathbb{N}}}
\newcommand{\xR}{{\mathbb{R}}}
\begin{document}
\title[A scheme for  the {\bf P}$_1$ radiative model]
{Analysis of a fractional-step scheme for the {\bf P}$_1$ radiative diffusion model}

\author{T. Gallou\"et}
\address{Universit\'e d'Aix-Marseille (AMU)}
\email{thierry.gallouet@univ-amu.fr}

\author{R. Herbin}
\address{Universit\'e d'Aix-Marseille (AMU)}
\email{raphaele.herbin@univ-amu.fr}

\author{A. Larcher}
\address{Institut de Radioprotection et de S\^{u}ret\'{e} Nucl\'{e}aire (IRSN)}
\email{larcher@kth.se}

\author{J.C. Latch\'e}
\address{Institut de Radioprotection et de S\^{u}ret\'{e} Nucl\'{e}aire (IRSN)}
\email{jean-claude.latche@irsn.fr}

\begin{abstract}
We address in this paper a nonlinear parabolic system, which is built to retain the main mathematical difficulties of the P1 radiative diffusion physical model.
We propose a finite volume fractional-step scheme for this problem enjoying the following properties.
First, we show that each discrete solution satisfies \apriori\ $\xLinfty$-estimates, through a discrete maximum principle; by a topological degree argument, this yields the existence of a solution, which is proven to be unique.
Second, we establish uniform (with respect to the size of the meshes and the time step) $\xLtwo$-bounds for the space and time translates; this proves, by the Kolmogorov theorem, the relative compactness of any sequence of solutions obtained through a sequence of discretizations the time and space steps of which tend to zero; the limits of converging subsequences are then shown to be a solution to the continuous problem.
Estimates of time translates of the discrete solutions are obtained through the formalization of a generic argument, interesting for its own sake.
\end{abstract}
\maketitle

%------------------------------------------------------------------------------------------------------------------------------------
%
\tableofcontents
%
%------------------------------------------------------------------------------------------------------------------------------------
%
\section{Introduction}

We address in this paper the following nonlinear parabolic system:
\begin{equation}
\left | \begin{array}{ll} \displaystyle
\frac{\partial u}{\partial t} - \lap u + u^4 -\varphi = 0
& \displaystyle \mbox{for } (x,t) \in \dom \times (0,T),
\\[1ex] \displaystyle
\varphi - \lap \varphi -u^4=0
& \displaystyle \mbox{for } (x,t) \in \dom \times (0,T),
\\[1ex] \displaystyle
u(x,0)=u_0(x)
& \displaystyle \mbox{for } x \in \dom,
\\[1ex] \displaystyle
u(x,t)=0
& \displaystyle \mbox{for } (x,t) \in \partial \dom \times (0,T),
\\[1ex] \displaystyle
\grad \varphi \cdot n =0
& \displaystyle \mbox{for } (x,t) \in \partial \dom \times (0,T),
\end{array} \right.
\label{pbcont}\end{equation}
where $\dom$ is a connected bounded subset of $\xR^d$, $d=2$ or $d=3$, which is supposed to be polygonal ($d=2$) or polyhedral ($d=3$), $T < \infty$ is the final time, $u$ and $\varphi$ are two real-valued functions defined on $\dom \times [0,T)$ and $\partial \dom$ stands for the boundary of $\dom$ of outward normal $n$.
The initial value for $u$, denoted by $u_0$, is supposed to lie in $\xLinfty(\dom) \cap \xHone_0(\dom)$ and to satisfy $u_0(x) \geq 0$ for almost every $x \in \dom$.

\medskip
This system of partial differential equations is inspired from a simplified radiative transfer physical model, the so-called {\bf P}$_1$ model, sometimes used in computational fluid dynamics for the simulation of high temperature optically thick flows, as encountered for instance in fire modelling (see e.g. \cite{siegel-howell,franck-07-app} for expositions of the theory or the documentation of the CFX or FLUENT commercial codes for a synthetic description).

In this context, the unknown $u$ stands for the temperature, $\varphi$ for the radiative intensity and the first equation is the energy balance.
System \eqref{pbcont} has been derived with the aim of retaining the main mathematical difficulties of the initial physical model; in particular, adding a convection term in the first equation would only require minor changes in the theory developed hereafter.

\medskip
In this short paper, we give a finite volume scheme for the discretization of \eqref{pbcont} and prove the existence and uniqueness of the discrete solution and its convergence to a solution of \eqref{pbcont}, thus showing that this problem indeed admits a solution, in a weak sense which will be defined.

% -------------------------------------------------------------------------------------------------------------------------

\section{The finite volume scheme}\label{sec:FV}

Even though the arguments developed in this paper are valid for any general admissible discretization in the sense of \cite[Definition 9.1, p. 762]{eym-00-fin}, we choose, for the sake of simplicity, to restrict the presentation to simplicial meshes.
We thus suppose given a triangulation $\mesh$ of $\dom$, that is a finite collection of d-simplicial control volumes $K$, pairwise disjoint, and such that $\displaystyle \bar \dom  =\cup_{K \in \mesh} \bar K$; the mesh is supposed to be conforming in the sense that two neighbouring simplices share a whole face (\ie\ there is no hanging node). 
In addition, we assume that, for any $K \in \mesh$, the circumcenter $x_K$ of $K$ lies in $K$; note that, for each neighbouring control volumes $K$ and $L$, the segment $[x_K,x_L]$ is orthogonal to the face $K|L$ separating $K$ from $L$.

\medskip
For each simplex $K$, we denote by $\edges(K)$ the set of the faces of $K$ and by $|K|$ the measure of $K$.
The set of faces of the mesh $\edges$ is split into the set $\edgesint$ of internal faces (\ie\ separating two control volumes) and the set $\edgesext$ of faces included in the domain boundary.
For each internal face $\edge=K|L$, we denote by $|\edge|$ the (d-1)-dimensional measure of $\edge$ and by $d_\edge$ the distance $d(x_K,x_L)$; for an external face $\edge$ of a control volume $K$, $d_\edge$ stands for distance from $x_K$ to $\edge$.
The regularity of the mesh is characterized by the parameter $\theta_\mesh$ defined by:
\begin{equation}
\theta_\mesh \eqdef \min_{K \in \mesh}\  \frac{\rho_K}{h_K}
\label{regul}\end{equation}
where $\rho_K$ and $h_K$ stands for the diameter of the largest ball included in $K$ and the diameter of $K$, respectively.

\medskip
We denote by $\xHdisc(\dom)\subset \xLtwo(\dom)$ the space of functions which are piecewise constant over each control volume $K\in\mesh$.
For all $u\in \xHdisc(\dom)$ and for all $K\in\mesh$, we denote by $u_K$ the constant value of $u$ in $K$.
The space $\xHdisc(\dom)$ is equipped with  the following classical Euclidean structure:
for $(u,v)\in (\xHdisc(\dom))^2$, we define the following inner product:
\begin{equation}\label{amudis}
[u,v]_\disc \eqdef
\sum_{\edge \in \edgesint\ (\edge=K|L)} \frac{|\edge|}{d_\edge}\ (u_L - u_K)(v_L - v_K)
+\sum_{\edge \in \edgesext\ (\edge \in \edges(K))} \frac{|\edge|}{d_\edge}\ u_K\ v_K .
\end{equation}
Thanks to the discrete Poincar\'e inequality \eqref{poindis} given below, this scalar product defines a norm on $\xHdisc(\dom)$:
\begin{equation}\label{norme_discrete}
\normeundisc{u} \eqdef  [u,u]_\disc ^{1/2}.
\end{equation}
The following discrete Poincar\'e inequality holds (see Lemma 9.1, p. 765, in \cite{eym-00-fin}):
\begin{equation}
\normedeu{u}{\dom} \le \diam(\dom)\ \normeundisc{u} \qquad
\forall u \in \xHdisc(\dom).
\label{poindis}\end{equation}
We also define the following semi-inner product and semi-norm:
\[
<u,v>_\disc \eqdef
\sum_{\edge \in \edgesint\ (\edge=K|L)} \frac{|\edge|}{d_\edge}\ (u_L - u_K)(v_L - v_K),
\qquad \snormeundisc{u} \eqdef <u,u>_\disc^{1/2}.
\]
These inner products can be seen as discrete analogues to the standard $\xHone$-inner product, with an implicitly assumed zero boundary condition in the case of the inner product $[\cdot,\cdot]_\disc$.
\\
For any function $u \in \xHdisc(\dom)$, we also define the following discrete $\xHn{{-1}}$-norm:
\[
\normemundisc{u} \eqdef \sup_{\substack{v \in \xHdisc(\dom) \\ v \neq 0}} \quad \frac 1 {\normeundisc{v}}  \int_\dom u\,v \dx.
\]
By inequality \eqref{poindis}, the $\normemundisc{\cdot}$-norm is controlled by the $\xLtwo(\dom)$-norm.

\medskip
The discrete Laplace operators associated with homogeneous Dirichlet and Neumann boundary conditions, denoted by $\lap_{\mesh,\DIR}$ and $\lap_{\mesh,\NEU}$ respectively, are defined as follows:
\begin{equation}
\begin{array}{l} \displaystyle
\forall \psi \in \xHdisc(\dom),
\\[2ex] \qquad \begin{array}{l} \displaystyle
(\lap_{\mesh,\NEU} (\psi))_K = \frac{1}{|K|} \sum_{\edge=K|L} \frac{|\edge|}{d_\edge}\ (\psi_L - \psi_K),
\\ \displaystyle
(\lap_{\mesh,\DIR} (\psi))_K=(\lap_{\mesh,\NEU} (\psi))_K+ \frac{1}{|K|} \sum_{\edge \in \edgesext\cap \edges(K)} \frac{|\edge|}{d_\edge}\ (-\psi_K).
\end{array}\end{array}
\label{laplacians}\end{equation}
The links between these operators and the above defined inner products is clarified by the following identities:
\[
\begin{array}{ll}\displaystyle
\forall \psi \in \xHdisc(\dom), \qquad 
& \displaystyle
\sum_{K\in\mesh} -|K|\ \psi_K\ (\lap_{\mesh,\NEU} (\psi))_K = <\psi,\psi>_\disc
\\[4ex] \displaystyle
\mbox{and}
& \displaystyle
\sum_{K\in\mesh} -|K|\ \psi_K\ (\lap_{\mesh,\DIR} (\psi))_K = [\psi,\psi]_\disc .
\end{array}
\]

\medskip
Finally, we suppose given a partition of the interval $(0,T)$, which we assume regular for the sake and simplicity, with $t^0=0,\dots, t^n=n\, \dt, \dots t^N=T$.

\medskip
A each time $t^n$, an approximation of the solution $(u^n,\varphi^n) \in \xHdisc(\dom) \times \xHdisc(\dom)$ is given by the following finite volume scheme:
\begin{subequations}
	\begin{align}
	&\forall K \in \mesh, \nonumber \\
	& \qquad  \frac{u_K^{n+1} -u_K^n}{\dt} - (\lap_{\mesh,\DIR} (u^{n+1}))_K + |u_K^{n+1}|\ (u_K^{n+1})^3 - \varphi_K^n=0,  \, \mbox{ for } n=0,\ldots, N-1, \label{pbdiscu}\\
	& \qquad  \varphi_K^{n+1} - (\lap_{\mesh,\NEU} (\varphi^{n+1}))_K - (u_K^{n+1})^4=0 \, \mbox{ for } n=0,\ldots, N-1,  \label{pbdiscphi}\\
	& \qquad u_K^0= \frac{1}{|K|} \int_K u_0(x)\,\dx,  \label{initialisation}\\
	& \qquad  \varphi_K^{0} - (\lap_{\mesh,\NEU} (\varphi^{0}))_K - (u_K^{0})^4=0  \label{initialisationphi}
\end{align}
\label{pbdisc}
\end{subequations}
Note that in \eqref{pbdiscu}, the term $u^4$ is discretized as $|u_K^{n+1}|\ (u_K^{n+1})^3$ to ensure positivity (see proof of Proposition \ref{linfty_est} and Remark \ref{qodd} below).
% -------------------------------------------------------------------------------------------------------------------------

\section{A priori $\xLinfty$ estimates, existence and uniqueness of the discrete solution}

We prove in this section that the numerical scheme \eqref{pbdisc} admits a unique solution.
The proof is based on a topological degree argument, which requires some a priori estimates on possible solutions to \eqref{pbdisc}.
Uniqueness is based on the positivity property of the scheme. 

\smallskip

\begin{proposition}[Existence and uniqueness of the approximate solution]~ 

\begin{enumerate}
\item The scheme \eqref{pbdisc} has a unique solution.
\item For $0 \leq n \leq N$, the unknown $\varphi^n$ satisfies the following estimate:
\[
\forall K\in \mesh,\quad 0 \leq \varphi_K^n \leq \left[ \max_{L \in \mesh} u_L^n \right]^4.
\]
\item For $1 \leq n \leq N$, the unknown $u^n$ satisfies the following estimate:
\[
\forall K\in \mesh,\quad 0 \leq u_K^n \leq \max_{L \in \mesh} u_L^{n-1}.
\]
\end{enumerate}
\label{linfty_est}\end{proposition}
\begin{proof}
\underline{Step one}: positivity of the unknowns.\\[1ex]
We first observe that, from its definition \eqref{initialisation} and thanks to the fact that $u_0$ is non-negative, $u^0$ is a non-negative function.
Let us suppose that this property still holds at time step $n$.
We write the second equation of the scheme \eqref{pbdiscphi} as:
\[
\varphi_K^n - (\lap_{\mesh,\NEU}(\varphi^n))_K= (u_K^n)^4.
\]
This set of relations is a linear system for $\varphi^n$, and the matrix of  this system is clearly an M-matrix: indeed, from the definition \eqref{laplacians} of $\lap_{\mesh,\NEU}$, it can be easily checked that its diagonal is strictly dominant and has only positive entries, and all its off-diagonal entries are non-positive.
Since the right-hand side of this equation is non-negative, $\varphi^n$ is also non-negative.
The first equation  \eqref{pbdiscu} of the scheme can now be recast as:
\begin{equation}
\left[ \frac 1 \dt + |u_K^{n+1}|\ (u_K^{n+1})^2 \right] u_K^{n+1} - (\lap_{\mesh,\DIR}(u^{n+1}))_K = \frac 1 \dt u_K^n + \varphi^n_K.
\label{first_eq_linear}\end{equation}
This set of relations can be viewed as a  (nonlinear) matrix system for the unknown $u^{n+1}$; the matrix of this system depends on $u^{n+1}$ but is also an M-matrix for any values of $u^{n+1}$.
Since we know that $\varphi^n\geq 0$, the right-hand side of this equation is by assumption non-negative and so is $u^{n+1}$.

\bigskip
\noindent \underline{Step two}: upper bounds.\\[1ex]
Let $\bar \varphi$ be a constant function of $\xHdisc(\dom)$.
From the definition \eqref{laplacians} of $\lap_{\mesh,\NEU}$, we see that $\lap_{\mesh,\NEU}(\bar \varphi)=0$;  the second relation of the scheme \eqref{pbdiscphi} thus implies:
\[
(\varphi^{n+1}-\bar \varphi)_K - (\lap_{\mesh,\NEU}(\varphi^{n+1}-\bar \varphi))_K= (u_K^{n+1})^4 - \bar \varphi_K.
\]
Choosing for the constant value of $\bar \varphi$ the quantity $(\max_{K\in\mesh} u_K^{n+1})^4$ yields a non-positive right-hand side, and so, from the above mentioned property of the matrix of this linear system, $(\varphi^{n+1}-\bar \varphi)_K \leq 0,\ \forall K \in \mesh$, which equivalently reads:
\begin{equation}
\varphi_K^{n+1} \leq (\max_{L\in\mesh} u_L^{n+1})^4, \quad \forall K \in \mesh.
\label{upp_phi}\end{equation}
Let us now turn to the estimate of the first unknown $u^{n+1}$.
Let $K_0$ be a control volume where $u^{n+1}$ reaches its maximum value.
From the definition \eqref{laplacians} of $\lap_{\mesh,\DIR}$, it appears that:
\[
-(\lap_{\mesh,\DIR}(u^{n+1}))_{K_0} \geq 0.
\]
The first  equation \eqref{pbdisc} of the scheme for the unknown $(u^{n+1}))_{K_0}$ reads:
\begin{equation}
\frac{1}{\dt}\,(u_{K_0}^{n+1}-u_{K_0}^n) - (\lap_{\mesh,\DIR}(u^{n+1}))_{K_0} + \left[ |u_{K_0}^{n+1}|\ (u_{K_0}^{n+1})^3 - \varphi_{K_0}^n \right]=0.
\label{upp_u}\end{equation}
By the inequality \eqref{upp_phi}, we see that supposing that $u_{K_0}^{n+1} > \max_{K\in\mesh} u_K^n$ yields that the first and third term of the preceding relation are positive, while the second one is non-negative, which is in contradiction with the fact that their sum is zero.

\bigskip
\noindent \underline{Step three}: existence of a solution.\\[1ex]
Let us suppose that we have obtained a solution to the scheme up to time step $n$.
Let the function $F$ be defined as follows:
\[
\left|\begin{array}{l} \begin{array}{lcl}
F :\xR^{\card(\mesh)} \times [0,1] & \longrightarrow & \xR^{\card (\mesh)},
\\[1ex] \displaystyle
(U,\alpha) = \left( (u_K)_{K \in \mesh},\ \alpha \right) & \mapsto &F(U,\alpha)   =\displaystyle (v_K)_{K \in \mesh}\quad \mbox{ such that:}
\end{array}
\\ \displaystyle \hspace{7ex}
\forall K \in \mesh, \quad v_K= \frac{1}{\dt}\,(u_K-u_K^n) - (\lap_{\mesh,\DIR}(u))_K + \alpha \left[ |u_K|\,(u_K)^3 - \varphi_K^n \right].
\end{array}\right.
\]
It is clear that  $u^{n+1}$ is a solution to  \eqref{pbdiscu} if and only if:
\begin{equation}
F((u^{n+1}_K)_{K\in \mesh},1)=0.
\label{sch_deg}\end{equation}
First, we observe that the function $F_0$, which maps $\xR^{\card (\mesh)}$ onto $\xR^{\card (\mesh)}$ and is defined by $F_0(U)=F(U,0)$, is affine and one-to-one.
Second, a straightforward adaptation of steps one and two yields that the estimates on any possible solution $u^{n+1}$ to  \eqref{sch_deg}  hold uniformly for $\alpha \geq 0$.
The existence of a solution to \eqref{sch_deg} then follows by a topological degree argument (see e.g. \cite{deimling}).

\noindent Finally, the existence (and uniqueness) of the solution to the second equation \eqref{pbdiscphi} of the scheme, which is a linear system, follows from the above mentioned properties of the associated matrix.

\bigskip
\noindent \underline{Step four:} uniqueness of the solution.\\[1ex]
Let us suppose that the solution is unique up to step $n$ and that there exist two solutions $u^{n+1}$ and $v^{n+1}$ to \eqref{pbdiscu}.
By the identity $a^4-b^4=(a-b)\,(a^3+a^2b+ab^2++b^3)$, the difference $\delta u= u^{n+1}-v^{n+1}$ satisfies the following system of equations:
\[
%\begin{array}{l} \displaystyle
\forall K \in \mesh,\quad 
\left[ \frac 1 \dt + \left((u_K^{n+1})^3 + (u_K^{n+1})^2 v_K^{n+1} + u_K^{n+1} (v_K^{n+1})^2 + (v_K^{n+1})^3\right) \right]
\ \delta u_K
%\qquad \\[3ex] \hfill \displaystyle
-(\lap_{\mesh,\DIR}(\delta u))_K=0.
%\end{array}
\]
Since we know from the precedent analysis that both $u_K^{n+1}$ and $v_K^{n+1}$ are non-negative, this set of relations can be seen as a matrix system for $\delta u$; the matrix of this system is again an M-matrix; we thus get $\delta u=0$, which proves the uniqueness of the solution.
\end{proof}
\begin{remark}[On the non-negativity]
We see from Relation \eqref{first_eq_linear} that the discretization of $u^4$ as a product of a positive quantity (here $|u^{n+1}|^p$) and $(u^{n+1})^q$ (where $p+q=4$) with $q$ odd (\ie\ $q=3$ or $q=1$) is essential to prove the non-negativity of $u^{n+1}$; note that we have indeed observed in practice some (non-physical) negative values when this term is discretized as $(u^{n+1})^4$.

We also wish to emphasize that the non-negativity is obtained here thanks to the fact that the discretization is performed with a two-point flux finite volume scheme, which yields an $M$-matrix for the discrete Laplace operator. 
Hence we have only considered here the admissible meshes in the sense of \cite[Definition 9.1, p. 762]{eym-00-fin}, since these meshes satisfy an orthogonality condition that yields the consistency of the  two-point flux approximation. 
On general meshes, several consistent schemes have been designed, but these  schemes do not ensure the positivity of the solution naturally; more on this subject may be found in e.g. \cite{eym-14-two}.
\label{qodd}\end{remark}
%
% -------------------------------------------------------------------------------------------------------------------------
%
\section{Convergence to a solution of the continuous problem}\label{sec:conv}

Let $\xHdiscd$ be the space of piecewise constant functions over each $K \times I^n$, for $K\in \mesh$ and $I^n=[t^n,\ t^{n+1}),\ 0\leq n \leq N-1$.
To each sequence $(u^n)_{n=0,N}$ of functions of $\xHdisc(\dom)$, we associate the function $u \in \xHdiscd$ defined by $
u(x,t)=u^n(x)$ for $t^n\leq t <t^{n+1},\ 0\leq n \leq N-1$.
In addition, for any $u \in \xHdiscd$, we define $\dtu \in \xHdiscd$ by 
\[
 \dtu(x,t)=\dtu^n(x) \mbox{ for } t^n\leq t <t^{n+1},\ 0\leq n \leq N-1
\]
 where the function $\dtu^n\in \xHdisc(\dom)$ is defined by:
\[
\dtu^n(x) \eqdef \frac{u^{n+1}(x)-u^n(x)}{\dt}
\qquad \bigl(\mbox{\ie\ } \dtu^n_K=\frac{u^{n+1}_K-u^n_K}{\dt},\quad \forall K \in \mesh \bigr).
\]

\bigskip

For any function $u \in \xHdiscd$, we define the following norms and semi-norms:
\[
\begin{array}{l} \displaystyle
\normLdHu{u}^2 \eqdef \dt\ \sum_{n=0}^N \normeundisc{u^n}^2,
\\[3ex] \displaystyle
\normLdHmu{u}^2 \eqdef \dt\ \sum_{n=0}^{N-1} \normemundisc{u^n}^2,
\\[3ex] \displaystyle
\snormLdHu{u}^2 \eqdef \dt\ \sum_{n=0}^N \snormeundisc{u^n}^2.
\end{array}
\]
The norms $\normLdHu{\cdot}$ and $\snormLdHu{\cdot}$ can be seen as discrete equivalents of the $\xLtwo(0,T;\xHone(\dom))$-norm, and $\normLdHmu{\cdot}$ may be considered as a discrete $\xLtwo(0,T;\xHn{{-1}}(\dom))$-norm.

\bigskip
Let us now derive some estimates on the solution to the scheme \eqref{pbdisc}, which will be useful to get some compactness on sequences of approximate solutions.
\begin{proposition}[Estimates in energy norms]\label{prop:energy_bound}
Let $u$ and $\varphi$ be the functions of $\xHdiscd$ associated to the solution $(u^n)_{0 \leq n \leq N} \in \xHdisc(\dom)^{N+1}$ and $(\varphi^n)_{0 \leq n \leq N} \in \xHdisc(\dom)^{N+1}$ of the scheme \eqref{pbdisc}.
Then the following estimate holds:
\begin{equation}
\begin{array}{l} \displaystyle
\normLdHu{u} + \normLdHmu{\dtu} + \normedeu{\varphi}{(0,T)\times \dom} 
\hspace{10ex}\\[2ex] \displaystyle \hfill
+ \snormLdHu{\varphi} + \normedeu{\dtphi}{(0,T)\times \dom} \leq c_{\rm e},
\end{array}
\label{energy_bound}\end{equation}
where the real number $c_{\rm e}$ only depends on $\dom$, the initial data $u_0$ and (as a decreasing function) on the parameter $\theta_\mesh$ characterizing the regularity of the mesh, defined by \eqref{regul}.
\end{proposition}
\begin{proof}
First, we recall that, by a standard reordering of the summations, we have, for any function $u \in \xHdisc(\dom)$:
\begin{align*}
 -\sum_{K\in\mesh} |K| \ u_K\ (\lap_{\mesh,\DIR}(u))_K=\normeundisc{u}^2,\\
-\sum_{K\in\mesh} |K| \ u_K\ (\lap_{\mesh,\NEU}(u))_K=\snormeundisc{u}^2.
\end{align*}
Multiplying Equation \eqref{pbdiscu} by $2\,\dt \, |K|\, u_K^{n+1}$, using the equality $2\,a\,(a-b)=a^2 +(a-b)^2-b^2$, summing over each control volume of the mesh and using the first of the preceding identities yields, for $0 \leq n \leq N-1$:
\[
\begin{array}{l} \displaystyle
\normedeu{u^{n+1}}{\dom}^2+\normedeu{u^{n+1}-u^n}{\dom}^2-\normedeu{u^n}{\dom}^2 + 2\,\dt\, \normeundisc{u^{n+1}}^2
\hspace{10ex}\\[2ex] \displaystyle \hfill
+ 2\,\dt\,\int_\dom (u^{n+1})^5 \dx = 2\,\dt\,\int_\dom \varphi^n u^{n+1} \dx.
\end{array}
\]
By Proposition \ref{linfty_est}, $u^{n+1}$ is non-negative and $\varphi^n_K\leq \bar u_0^4,\ \forall K \in \mesh$, where $\bar u_0$ stands for $\max_{K\in\mesh} u_K^0$ and is thus bounded by the $\xLinfty$-norm of on the initial data $u_0$.
Therefore, we get:
\[
\normedeu{u^{n+1}}{\dom}^2-\normedeu{u^n}{\dom}^2 + 2\, \dt\, \normeundisc{u^{n+1}}^2
\leq 2\,\dt\, \bar u_0^4\ \int_\dom u^{n+1} \dx.
\]
By Cauchy-Schwarz' inequality, Young's inequality and the discrete Poincar\'e inequality \eqref{poindis}, we thus obtain:
\[
\normedeu{u^{n+1}}{\dom}^2-\normedeu{u^n}{\dom}^2 + \dt\, \normeundisc{u^{n+1}}^2
\leq \dt \, |\dom|\, \diam(\dom)^2\, \bar u_0^8.
\]
Summing from $n=0$ to $n=N-1$, we get:
\[
\normedeu{u^N}{\dom}^2 + \sum_{n=1}^N \dt \, \normeundisc{u^n}^2 \leq T \, |\dom|\, \diam(\dom)^2\, \bar u_0^8 
+ \normedeu{u^0}{\dom}^2.
\]
Since, by assumption, $u_0 \in \xHone_0(\dom)$, the discrete $\xHone$ norm of $u^0$, $\normeundisc{u^0}$ is bounded by $c \| u_0\|_{\xHone(\dom)}$ where the real number $c$ only depends on $\dom$ and, in a decreasing way, on the parameter $\theta_\mesh$ characterizing the regularity of the mesh (see e.g. Lemma 3.3 in \cite{eym-09-conv}).
Together with the preceding relation, this provides the control of the first term in \eqref{energy_bound}.

\medskip
We now turn to the estimate of $\normLdHmu{\dtu}$.
Let $v$ be a function of $\xHdisc(\dom)$; multiplying  \eqref{pbdiscu} by $|K|\,v_K $ and summing over $K \in \mesh$, we get for $0 \leq n \leq N-1$:
\[
\int_\dom \dtu^n \ v \dx= -[u^{n+1},v]_\disc - \int_\dom \left[ (u^{n+1})^4-\varphi^n \right] \, v \dx.
\]
Both $(u^{n+1})^4$ and $\varphi^n$ are non-negative functions which are bounded by $\bar u_0^4$, aand therefore the difference $(u^{n+1})^4-\varphi^n$ is itself bounded by $\bar u_0^4$; using the Cauchy-Schwarz inequality and the discrete Poincar\'e inequality \eqref{poindis}, we  thus get:
\[
\int_\dom \dtu^n \ v \dx
\leq \left[ \normeundisc{u^{n+1}} + \bar u_0^4 \, |\dom|^{1/2} \, \diam(\dom)\right] \normeundisc{v},
\]
and so:
\[
\normemundisc{\dtu^n} \leq \normeundisc{u^{n+1}} + \bar u_0^4 \, |\dom|^{1/2} \, \diam(\dom),
\]
which, by the bound of $\normLdHu{u}$, yields the control of the second term in \eqref{energy_bound}.

\medskip
As far as $\varphi$ is concerned, the second equation  \eqref{pbdiscphi} of the scheme and the initialization \eqref{initialisation} yield for $0 \leq n \leq N$:
\[
\normedeu{\varphi^n}{\dom}^2 + \snormeundisc{\varphi^n}^2  \leq \int_\dom \bar u_0^4 \ \varphi^n \dx.
\]
Thus, by Young's inequality, we get:
\[
\frac 1 2 \ \normedeu{\varphi^n}{\dom}^2 + \snormeundisc{\varphi^n}^2 \leq \frac 1 2 \, |\dom| \, \bar u_0^8.
\]
Multiplying by $\dt$ and summing over the time steps, this provides the estimates of $\normedeu{\varphi}{(0,T)\times \dom}$ and $\snormLdHu{\varphi}$ we are seeking.

\medskip
To obtain a control on $\dtphi$, we need a sharper estimate on $\dtu$.
Our starting point is once again Equation \eqref{pbdisc}-$(i)$, which we multiply this time by $|K|\ \dtu^n$ before summing over $K\in\mesh$, to get for $0 \leq n \leq N-1$, once again by the identity $a^2 -b^2 \leq a^2 +(a-b)^2-b^2=2\,a\,(a-b)$ and invoking the Cauchy-Schwarz inequality:
\[
\normedeu{\dtu^n}{\dom}^2 + \frac{1}{2 \, \dt}\ \left( \normeundisc{u^{n+1}}^2 - \normeundisc{u^n}^2\right)
\leq |\dom|^{1/2}\, \bar u_0^4 \ \normedeu{\dtu^n}{\dom},
\]
so, by Young's inequality:
\[
\normedeu{\dtu^n}{\dom}^2 + \frac{1}{\dt}\ \left( \normeundisc{u^{n+1}}^2 - \normeundisc{u^n}^2\right)
\leq |\dom|\, \bar u_0^8.
\]
Multiplying by $\dt$ and summing over the time steps yields:
\begin{equation}
\normedeu{\dtu}{(0,T) \times \dom}^2 + \normeundisc{u^N}^2 \leq |\dom|\, T \, \bar u_0^8 + \normeundisc{u^0}^2,
\label{eq:est_dtu}\end{equation}
which provides an estimate for $\normedeu{\dtu}{(0,T) \times \dom}$.
Taking now the difference of the second equation \eqref{pbdiscphi} of the scheme at two consecutive time steps and using \eqref{initialisation} for the first step, we obtain for $0 \leq n \leq N-1$:
\[
\begin{array}{ll}
\forall K \in \mesh, \qquad
& \displaystyle
\dtphi_K^n - (\lap_{\mesh,\NEU}(\dtphi^n))_K=
\frac {(u_K^{n+1})^4-(u_K^n)^4} \dt
\\[2ex] & \displaystyle \hspace{7ex}
=\left[(u_K^{n+1})^3 + (u_K^{n+1})^2\,u_K^n +u_K^{n+1}\, (u_K^n)^2 +(u_K^n)^3 \right] \dtu_K^n.
\end{array}
\]
Multiplying by $|K|\ \dtphi_K^n$ over each control volume of the mesh and summing yields:
\[
\normedeu{\dtphi^n}{\dom}^2 + \snormeundisc{\dtphi^n}^2 \leq 4 \bar u_0^3\  \normedeu{\dtu^n}{\dom}\ \normedeu{\dtphi^n}{\dom},
\]
which, using Young's inequality, multiplying by $\dt$ and summing over the time steps yields the desired estimate for $\normedeu{\dtphi}{(0,T)\times \dom}$, thanks to \eqref{eq:est_dtu}.
\end{proof}

\bigskip
Let us now prove that sequences of solutions of the scheme  \eqref{pbdisc} converge, as the space and time steps tend to 0 (up to a subsequence),  to a weak solution of Problem \eqref{pbcont}, thereby proving the existence of this weak solution.

\begin{theorem}[Convergence of the finite volume scheme]
Let $(u^{(m)})_{m \in \xN}$ and $(\varphi^{(m)})_{m \in \xN}$ be a sequence of solutions to \eqref{pbdisc} with a sequence of discretizations such that the space and time step, $h^{(m)}$ and $\dt^{(m)}$ respectively, tends to zero.
We suppose that the parameters $\theta_{\mesh^{(m)}}$ characterizing the regularity of the meshes of this sequence are bounded away from zero, \ie\ $\theta_{\mesh^{(m)}} \geq \theta >0,\ \forall\, m \in \xN$.
Then there exists a subsequence, still denoted by $(u^{(m)})_{m \in \xN}$ and $(\varphi^{(m)})_{m \in \xN}$ and two functions $\tilde u$ and $\tilde \varphi$ such that:
\begin{enumerate}
\item $u^{(m)}$ and $\varphi^{(m)}$ tends to $\tilde u$ and $\tilde \varphi$ respectively in $\xLtwo((0,T)\times \dom)$,
\item $\tilde u$ and $\tilde \varphi$ are solution to the continuous problem \eqref{pbcont} in the following weak sense:
\end{enumerate}
\begin{subequations}
\begin{align}
& \tilde u \in \xLinfty((0,T)\times \dom) \cap \xLtwo(0,T ;\, \xHone_0(\dom)),
\ \tilde \varphi \in \xLinfty((0,T)\times \dom) \cap \xLtwo(0,T ;\, \xHone(\dom))
\nonumber \\
& \int_{0,T} \int_\dom \left[ -\frac{\partial \psi}{\partial t}\, \tilde u + \grad \tilde u \cdot \grad \psi 
+ (\tilde u^4-\tilde \varphi)\, \psi \right]\dx \ddt =
 \int_\dom \psi(x,0) \, u_0(x) \dx, \forall \psi \in \xCinfty_c([0,T)\times \dom), \label{pbweakfirst}\\
& \int_{0,T} \int_\dom \left[(\tilde \varphi- \tilde u^4) \, \psi + \grad \tilde \varphi \cdot \grad \psi \right] \dx \ddt =0, \forall \psi \in \xCinfty_c([0,T)\times \bar \dom).  \label{pbweaksecond}
 \end{align}
\label{pbweak}\end{subequations}
\end{theorem}
\begin{proof}
The following estimates of the space translates can be found in \cite{eym-00-fin}, Lemma 9.3, p. 770 and Lemma 18.3, p. 851:
\[
\begin{array}{l}
\forall v \in \xHdisc,\ \forall \eta \in \xR^d,
\\[2ex]\displaystyle \hspace{10ex}
\normedeu{\hat v(\cdot + \eta) - \hat v }{\xR^d}^2 \leq
\normeundisc{v}^2\ |\eta| \ \left[|\eta| + c(\dom)\, h \right],
\\[2ex]\displaystyle \hspace{10ex}
\normedeu{\hat v(\cdot + \eta) - \hat v }{\xR^d}^2 \leq
|\eta| \ \left[ \snormeundisc{v}^2\ (|\eta| + 2\, h) + 2\,|\partial \dom| \, \Vert v \Vert_{\xLinfty(\dom)} \right],
\end{array}
\]
where $\hat v$ stands for the extension by zero of $v$ to $\xR^d$ and the real number $c(\dom)$ only depends on the domain.

\medskip
For $m$ given, let $\hat u^{(m)}$ and $\hat \varphi^{(m)}$ be the functions of $\xLinfty(\xR^d \times \xR)$ obtained by extending  $u^{(m)}$ and $\varphi^{(m)}$ by $0$ from $\dom \times [0,T)$ to $\xR^d \times \xR$.
The estimates of Proposition \ref{prop:energy_bound} of $\normLdHu{u^{(m)}}$ and $\snormLdHu{\varphi^{(m)}}$, together with the $\xLinfty$-bound for $\varphi^{(m)}$, thus allow to bound independently of $m$ the space translates of $\hat u^{(m)}$ and $\hat \varphi^{(m)}$ in the $\xLtwo((0,T)\times\dom)$-norm.
In addition, Theorem \ref{th:time_trans} applied with $\normes{\cdot}$ equal to $\normeundisc{\cdot}$ for $u^{(m)}$ and with $\normes{\cdot}$ equal to $\Vert \cdot \Vert_{\xLtwo(\dom)}$ for $\varphi^{(m)}$, together with the estimates of $\normLdHu{u^{(m)}}$, $\normLdHmu{\dtu^{(m)}}$, $\normedeu{\varphi}{(0,T)\times \dom}$ and $\normedeu{\dtphi}{(0,T)\times \dom}$ of Proposition \ref{prop:energy_bound} allows to bound the time translates of $\hat u^{(m)}$ and $\hat \varphi^{(m)}$, still independently of $m$ and in the $\xLtwo((0,T)\times\dom)$-norm.
Furthermore, by Proposition \ref{linfty_est}, the sequences $(\hat u^{(m)})_{m \in \xN}$ and $(\hat \varphi^{(m)})_{m \in \xN}$ are uniformly bounded in $\xLinfty((0,T)\times\dom)$, and thus in $\xLtwo((0,T)\times\dom)$.
By Kolmogorov theorem (see e.g. \cite{eym-00-fin}, Theorem 14.1, p. 833), these sequences are relatively compact and strongly converge in $\xLtwo((0,T)\times\dom)$ to, respectively, $\tilde u$ and $\tilde \varphi$.

\medskip

The regularity of the limits, {\it i.e.} the fact that $\tilde u$ (resp. $\tilde \varphi$) lies  in $\xLtwo(0,T ;\, \xHone_0(\dom))$  (resp. $\xLtwo(0,T ;\, \xHone(\dom))$ can be proven by  invoking  \cite[Theorem 14.2 and Theorem 14.3]{eym-00-fin} thanks to  the uniform bounds on $\normLdHu{u^{(m)}}$ and $\snormLdHu{\varphi^{(m)}}$ given in Proposition \ref{prop:energy_bound}, similarly to  the technique of  \cite{eym-07-conv}; it can also be proven (again thanks to the estimates of Proposition \ref{prop:energy_bound}) by invoking  Theorem \ref{theo-regulimit} given in the appendix  below, which gives a more straightforward way to obtain the regularity in a more general setting. 
The spaces which are considered in the application of  Theorem \ref{theo-regulimit} are $B= L^2(\Omega)$, $X_m = H_{\mesh^{(m)}}$ and $X= H^1_0(\Omega)$.

\medskip
To prove that $\tilde u$ and $\tilde \varphi$ are solution to the continuous problem, it remains to prove that \eqref{pbweak} holds.
This proof is rather standard (see e.g. the proof of Theorem 18.1 pp. 858--862 in \cite{eym-00-fin} for a similar, although more complicated, problem) and we only give here the main arguments.
Let $\psi$ be a function of $\xCinfty_c([0,T)\times \dom)$.
We define $\psi_K^n$ by $\psi_K^n=\psi(x_K,t^n)$.
Multiplying the  \eqref{pbdiscu} by $\dt \, |K|\ \psi_K^{n+1}$ and summing up over the control volumes and the time steps, we get for any element of the sequence of discrete solutions:
\[
\begin{array}{r}\displaystyle
\sum_{n=0}^{N-1} \sum_{K\in\mesh} \dt \, |K|\ \psi_K^{n+1}
\left[\frac{u_K^{n+1} -u_K^n}{\dt} - (\lap_{\mesh,\DIR} (u^{n+1}))_K + (u_K^{n+1})^4 - \varphi_K^n\right]
\hspace{7ex}\\ \displaystyle
=T_1 +T_2 + T_3 =0,
\end{array}
\]
where, for enhanced readability, the superscript $^{(m)}$ has been omitted and:
\[
\left| \quad \begin{array}{l} \displaystyle
T_1= \sum_{n=0}^{N-1} \sum_{K\in\mesh} |K|\ \psi_K^{n+1}\ \left[ u_K^{n+1} -u_K^n \right],
\\[3ex] \displaystyle
T_2=\sum_{n=0}^{N-1} \sum_{K\in\mesh} - \dt \, |K|\ \psi_K^{n+1}\ (\lap_{\mesh,\DIR} (u^{n+1}))_K,
\\[3ex] \displaystyle
T_3=\sum_{n=0}^{N-1} \sum_{K\in\mesh} \dt \, |K|\ \psi_K^{n+1}\ \left[ (u_K^{n+1})^4 - \varphi_K^n \right].
\end{array} \right.
\]
Reordering the summations and using the fact that $\psi(\cdot,T)=0$, we get for $T_1$:
\[
T_1=-\sum_{K\in\mesh} |K|\ \psi_K^1\ u_K^0 + \sum_{n=1}^{N-1} \sum_{K\in\mesh} |K|\ u_K^n\ \left[ \psi_K^n -\psi_K^{n+1} \right].
\]
The first term of the right hand side reads:
\[
\begin{array}{ll} \displaystyle
T_{1,1}=-\int_{\dom} u_{0}(x)\, \psi(x,0)\dx
& \displaystyle
+ \sum_{K\in\mesh} \int_K (u_0(x)-u^0_K) \, \psi(x,0) \dx
\\  & \displaystyle
+\sum_{K\in\mesh} |K|\ u_K^0 \ \underbrace{\left[ \frac{1}{|K|} \int_K \psi(x,0)\dx - \psi(x_K,\dt) \right]}_{\displaystyle R_\psi}.
\end{array}
\]
On one hand, $u^0$ converges to $u_0$ in $\xLn{1}{(\dom)}$ and $\psi(\cdot, 0) \in \xCinfty_c(\dom)$, so the second term of $T_{1,1}$ tends to zero with $h$; on the other hand, since $u_K^0 \leq \bar u_0 $, $\forall K \in \mesh$ and, from the regularity of $\psi$, $|R_\psi|\leq c_\psi \ (\dt +h)$, the third term of $T_{1,1}$ also tends to zero with $\dt$ and $h$.
Let us now turn to the second term in the expression of $T_1$:
\[
\begin{array}{l}\displaystyle
T_{1,2}
=\sum_{n=1}^{N-1} \sum_{K\in\mesh} |K|\ u_K^n\ \left[ \psi_K^n -\psi_K^{n+1} \right]
=-\int_{\dt}^T \int_\dom u(x,t)\ \frac{\partial \psi}{\partial t}(x,t) \dx \ddt
\hspace{5ex}\\ \hfill \displaystyle
+\sum_{n=1}^{N-1} \sum_{K\in\mesh} \dt\, |K|\ u_K^n\ (R_\psi)_K^n
\end{array}
\]
with:
\[
(R_\psi)_K^n= \frac 1 {\dt\, |K|}\int_{t^n}^{t^{n+1}} \int_K \frac{\partial \psi}{\partial t}(x,t) \dx \ddt 
         - \frac{\psi(x_K,t^{n+1})-\psi(x_K,t^n)}{\dt},
\]
and thus $|(R_\psi)_K^n|\leq c_\psi \ (\dt +h)$.
Since $u_K^n \leq \bar u_0$, $\forall K \in \mesh$ and $0 \leq n \leq N$, we get:
\[
T_1 \to - \int_{\dom} u_{0}(x)\, \psi(x,0)\dx - \int_0^T \int_\dom \tilde u(x,t)\ \frac{\partial \psi}{\partial t}(x,t) \dx \ddt
\qquad \mbox{as} \quad m \to \infty.
\]
Reordering the summations in $T_2$ and using the fact that $\psi(\cdot,T)=0$, we obtain:
\[
\begin{array}{ll} \displaystyle
T_2
=\sum_{n=0}^{N-2} \sum_{K\in\mesh} - \dt \, |K|\ u_K^{n+1}\ (\lap_{\mesh,\DIR} (\psi^{n+1}))_K
\\ \displaystyle \quad
=-\int_{\dt}^T \int_\dom u(x,t)\ \lap \psi (x,t) \dx \ddt + \sum_{n=0}^{N-2} \sum_{K\in\mesh} - \dt \, |K|\ u_K^{n+1} \sum_{\edge \in \edges(K)}|\edge|\ (R_\psi)_\edge^{n+1},
\end{array}
\]
where the residual term $(R_\psi)_\edge^{n+1}$ is the difference of the mean value of $\grad \psi\cdot n$ over $\edge \times (t^{n+1},t^{n+2})$ and its finite volume approximation.
The fact that the second term in the right hand side of this relation tends to zero is thus a classical consequence of the control of $\normLdHu{u}$ and the consistency of the diffusive fluxes $(R_\psi)_\edge^{n+1}$ (see Theorem 9.1, pp. 772--776, in \cite{eym-00-fin}) and yields, as $\tilde u$ is known to belong to $\xHone_0(\dom)$:
\[
T_2 \to \int_0^T \int_\dom \grad \tilde u(x,t) \cdot \grad \psi(x,t) \dx \ddt
\qquad \mbox{as} \quad m \to \infty.
\]
Finally, $T_3$ reads:
\[
\begin{array}{l} \displaystyle
T_3=\int_{\dt}^T \int_\dom \psi(x,t)\ \left[ u(x,t)^4 - \varphi(x,t-\dt) \right]
\hspace{15ex}\\ \hfill \displaystyle
-\sum_{n=0}^{N-2} \sum_{K\in\mesh} \dt \, |K|\ \left[ (u_K^{n+1})^4 - \varphi_K^n \right]\ (R_\psi)_K^{n+1}
\end{array}
\]
with:
\[
(R_\psi)_K^{n+1}=\frac 1 {\dt \, |K|} \int_{t^{n+1}}^{t^{n+2}} \int_K \psi(x,t)\dx \ddt - \psi(x_K,t^{n+1}).
\]
The second term tends to zero by the $\xLinfty$-estimates for $u$ and $\varphi$ and the regularity of $\psi$.
Since $u^{(m)}$ tends to $\tilde u$ in $\xLtwo((0,T)\times \dom)$ and is bounded in $\xLinfty((0,T)\times \dom)$, $u^{(m)}$ converges to $\tilde u$ in $\xLtwo(0,T;\,\xLn{p}(\dom))$, for any $p \in [1,+\infty)$; in addition, from the time translates estimates, $\varphi^{(m)}(\cdot,\cdot-\dt) $ also converge to $\tilde \varphi$.
We thus get:
\[
T_3 \to \int_0^T \int_\dom \psi(x,t)\ \left[ \tilde u(x,t)^4 - \tilde \varphi(x,t) \right] \dx \ddt
\qquad \mbox{as} \quad m \to \infty.
\] 
Gathering the results for $T_1$, $T_2$ and $T_3$, we obtain \eqref{pbweakfirst}.
The  relation \eqref{pbweaksecond} is obtained using the same arguments; the convergence of the diffusion term in case of Neumann boundary conditions poses an additional difficulty which is solved in Theorem 10.3, pp. 810--815, in \cite{eym-00-fin}.
\end{proof}

\bigskip
The uniqueness of the solution to the problem under consideration is beyond the scope of the present paper.
Note however that such a result would imply, by a standard argument, the convergence of the whole sequence to the solution.
%
% -------------------------------------------------------------------------------------------------------------------------
%
\section{Conclusion}

We propose in this paper a finite volume scheme for a problem capturing the essential difficulties of a simple radiative transfer model (the so-called {\bf P}$_1$ model), which enjoys the following properties: the discrete solution exists, is unique, and satisfies a discrete maximum principle; in addition, it converges (possibly up to the extraction of a subsequence) to a solution of the continuous problem, which yields, as a by-product, that such a solution indeed exists.
For the proof of this latter result, we state and prove an abstract estimate allowing to bound the time translates of a finite volume discrete function, as a function of (possibly discrete) norms of the function itself and of its discrete time derivative; although this estimate is underlying in some already available analysis on finite volumes \cite[Chapter IV]{eym-00-fin}, \cite{eym-01-app,eym-07-conv} and  discontinous Galerkin approximations \cite{DipietroErn2010NavierStokes}, the present formulation is new and should be useful to tackle new problems.
Variants of the presented numerical scheme are now successfully running for the modelling of radiative transfer in the ISIS free software \cite{isis} developed at IRSN and devoted to the simulation of fires in confined buildings \cite{bab-05-ons}, and in particular nuclear power plants.
%
% -------------------------------------------------------------------------------------------------------------------------
%

\appendix
\section{Estimation of time translates}

The objective of this appendix is to state and prove an abstract result allowing to bound the time translates of a discrete solution.
\bigskip
We now introduce some notations.
Let $\xHdisc(\dom)$ and $\xHdiscd$ be the discrete functional spaces introduced in section \ref{sec:FV} and \ref{sec:conv} respectively.
We suppose given a norm $\normes{\cdot}$ on $\xHdisc(\dom)$, over which we also define the dual norm $\normems{\cdot}$ with respect to the $\xLtwo$-inner product:
\[
\forall u \in \xHdisc(\dom), \qquad
\normems{u} \eqdef \sup_{v \in \xHdisc(\dom),\, v \neq 0} \quad \frac{\displaystyle \int_\dom u\,v \dx}{\normes{v}}.
\]
These two spatial norms may be associated to a corresponding norm on $\xHdiscd$ as follows:
\[
\begin{array}{ll} \displaystyle
\forall u \in \xHdiscd,\ u=(u^n)_{0\leq n\leq N}, \qquad
& \displaystyle
\normLds{u}^2 = \sum_{n=0}^N \dt\ \normes{u^n}^2,
\\[3ex] \hfill
\mbox{and} \qquad
& \displaystyle
\normLdms{u}^2 = \sum_{n=0}^{N-1} \dt\ {\normems{u^n}}^2.
\end{array}
\]
 
%
% estimation des translations :
%
\begin{theorem}\label{th:time_trans}
Let $u$ be a function of $\xHdiscd$ and $\tau$ a real number.
We denote by $\hat u$ the extension by zero of $u$ to $\xR^d \times \xR$.
Then we have:
\[
\begin{array}{r}
\Vert \hat u(\cdot, \cdot +\tau) - \hat u(\cdot, \cdot)\Vert_{\xLtwo(\xR^d \times \xR)}^2
\leq \tau\ \left[ 2\ \normLds{u}^2 \right.
\hspace{25ex} \\[1.5ex] \displaystyle
\left. + \frac 1 2 \ \normLdms{\dtu}^2 +2\ \Vert u \Vert_{\xLinfty(0,T;\, \xLtwo(\dom))}^2 \right].
\end{array}
\]
\end{theorem}
\begin{proof}
Let $u$ be a function of $\xHdiscd$ and $t \in \xR$.
Let $\tau$ be a real number that we suppose positive.
The following identity holds:
\[
\hat u(\cdot, t+\tau)-\hat u(\cdot, t) = \chi^0_\tau(t)\ u^0 + \sum_{n=1}^{N-1} \chi^n_\tau(t) \left[ u^n - u^{n-1} \right]
- \chi^N_\tau(t)\ u^{N-1}.
\]
For $s \in \xR$ we define $n(s)$ by: $n(s)=-1$ if $s <0$, $n(s)$ is the index such that $t^{n(s)} \leq s <  t^{n(s)+1}$ for $0\leq s < t^N$, $n(s)=N+1$ for $t \geq t^N$.
Let $n_0(t)$ and $n_1(t)$ be given by $n_0(t)=n(t)$, $n_1(t)=n(t+\tau)$.
We adopt the convention $u^{-1}=u^{N}=0$.
With this notation, we have for $u(\cdot, t+\tau)-u(\cdot, t)$ the following equivalent expression:
\[
u(\cdot, t+\tau)-u(\cdot, t) = u^{n_1(t)}-u^{n_0(t)},
\]
and thus:
\[
\begin{array}{l} \displaystyle
\int_\dom \left[ u(x,t+\tau)-u(x,t)\right]^2 \dx =
\\ \displaystyle \hspace{5ex}
\int_\dom \left[u^{n_1(t)}-u^{n_0(t)}\right] 
\ \left[ \chi^0_\tau(t)\ u^0 + \sum_{n=1}^{N-1} \chi^n_\tau(t) \left[ u^n - u^{n-1} \right]
- \chi^N_\tau(t)\ u^{N-1} \right] \dx.
\end{array}
\]
Developping, we get:
\[
\int_\dom \left[ u(x,t+\tau)-u(x,t)\right]^2 \dx = T_1(t) + T_2(t) + T_3(t)
\]
with:
\[
\left| \quad \begin{array}{l} \displaystyle
T_1(t) =\chi^0_\tau(t) \int_\dom \left[u^{n_1(t)}-u^{n_0(t)}\right] \ u^0 \dx,
\\ \displaystyle
T_2 (t) = \sum_{n=1}^{N-1} \chi^n_\tau(t) \int_\dom \left[u^{n_1(t)}-u^{n_0(t)}\right] \left[ u^n - u^{n-1} \right] \dx,
\\ \displaystyle
T_3 (t) = - \chi^N_\tau(t)\int_\dom \left[u^{n_1(t)}-u^{n_0(t)}\right]\ u^{N-1} \dx.
\end{array} \right.
\]
We first estimate the integral of $T_1(t)$ over $\xR$.
Since $\chi^0_\tau(t)$ is equal to $1$ in the interval $[-\tau,0)$ and $0$ elsewhere, and since $u^{n_0(t)}=0$ for any negative $t$, we get:
\[
\int_\xR T_1(t) \ddt = \int_{-\tau}^0 \int_\dom u^{n_1(t)} \ u^0 \dx \ddt \leq \tau\ \Vert u \Vert_{\xLinfty(0,T;\, \xLtwo(\dom))}^2.
\]
By the same arguments, we get the same bound for the integral of $T_3(t)$:
\[
\int_\xR T_3(t) \ddt \leq \tau\ \Vert u \Vert_{\xLinfty(0,T;\, \xLtwo(\dom))}^2.
\]
From the definition of the $\normems{\cdot}$ norm, we get:
\[
T_2(t)
\leq \dt \sum_{n=1}^{N-1} \chi^n_\tau(t)\quad \normems{\dtu^{n-1}}\  \normes{u^{n_1(t)}-u^{n_0(t)}},
\]
and thus, by Young's inequality:
\[
T_2(t)
\leq \dt \sum_{n=1}^{N-1} \chi^n_\tau(t)\quad \left[ \frac 1 2 \ {\normems{\dtu^{n-1}}}^2 
+ \normes{u^{n_0(t)}}^2 + \normes{u^{n_1(t)}}^2 \right].
\]
Integrating over the time, we get:
\[
\int_\xR T_2(t) \ddt\leq T_{2,1} + T_{2,2} + T_{3,3},
\]
with 
\begin{align*}
 &T_{2,1}=\frac \dt 2 \int_\xR \sum_{n=1}^{N-1} \chi^n_\tau(t)\quad {\normems{\dtu^{n-1}}}^2 \ddt, \\
&T_{2,2}=\dt \sum_{m=0}^{N-1} \left[ \int_{t^m}^{t^{m+1}} \sum_{n=1}^{N-1} \chi^n_\tau(t) \ddt \right] \normes{u^m}^2, \\
& T_{2,3}=\dt \sum_{m=0}^{N-1} \left[ \int_{t^m-\tau}^{t^{m+1}-\tau} \sum_{n=1}^{N-1} \chi^n_\tau(t) \ddt \right] \normes{u^m}^2.
\end{align*}
Thans to Relation $(i)$ of the technical lemma \ref{int_t} given below, we have
\[
\begin{array}{ll}\displaystyle
T_{2,1}=& \displaystyle
= \frac \tau 2 \sum_{n=0}^{N-2} \dt\ {\normems{\dtu^n}}^2 
\\[3ex] & \displaystyle \leq \frac \tau 2 \ \normLdms{\dtu}^2.
\end{array}
\]
Since $u^{n_0(t)}=u^m$ for $t^m \leq t < t^{m+1}$, and thanks to Relation \eqref{inttii} of  Lemma \ref{int_t} we get:
\[
T_{2,2}\leq \tau \sum_{m=0}^{N-1}\dt\ \normes{u^m}^2=\tau \ \normLds{u}^2.
\]
Finally, $u^{n_1(t)}=u^m$ for $t^m-\tau \leq t < t^{m+1}-\tau$, and thus, by the same argument:
\[
T_{2,3}
\leq \tau \sum_{m=0}^{N-1}\dt\ \normes{u^m}^2=\tau \ \normLds{u}^2.
\]
This concludes the proof for positive $\tau$.
The case of negative $\tau$ follows by symmetry.
\end{proof}

\begin{lemma}\label{int_t}
Let $(t^n)_{0 \leq n \leq N}$ be such that $t^0=0$, $t^n=n\dt$, $t^N=T$, $\tau$ be a positive real number and $\chi^n_\tau:\ \xR \rightarrow \xR$ be the function defined by $\chi^n_\tau (t)=1$ if $t < t^n \leq t+\tau$ and $\chi^n_\tau (t)=0 $ otherwise.
Then, for any family of real numbers $(\alpha_n)_{n=1,N}$ and, respectively, for any real number $t$, we have the following identities:
\begin{align}
\int_\xR \left[ \sum_{n=1}^N \alpha_n \chi^n_\tau(t) \right] \ddt= \tau \sum_{n=1}^N \alpha_n,\label{intti} \\ 
\int_t ^{t+\dt}\left[ \sum_{n=1}^N \chi^n_\tau(s) \right] \ds \leq \tau.\label{inttii}
\end{align}
 
\end{lemma}
\begin{proof}
The function $\chi^n_\tau(t)$ is equal to one for $t \in [ t^n-\tau ,\ t^n)$, so we have:
\[
\int_\xR \left[ \sum_{n=1}^N \alpha_n \chi^n_\tau(t) \right] \ddt= 
\sum_{n=1}^N \alpha_n \int_{t^n-\tau}^{t^n} \ddt.
\]
To obtain the inequality \eqref{inttii} we remark that $t \in [ t^n-\tau ,\ t^n)$ is equivalent to $t-t^n \in [-\tau ,\ 0)$ and so $\chi^n_\tau(t)=1$ is equivalent to $\chi^0_\tau(t-t^n)=1$ (under the assumption $t^0=0$).
We thus have:
\[
\int_t ^{t+\dt}\left[ \sum_{n=1}^N \chi^n_\tau(s) \right] \ds=
\sum_{n=1}^N \int_{t-t^n} ^{t-t^n+\dt} \chi^0_\tau(s) \ds \leq \int_\xR \chi^0_\tau(s) \ds = \tau.
\]
\end{proof}

Let us remark that the above theorem may be generalized to the Banach setting as in \cite{eym-07-conv}. 
In fact, this has led to some adaptations of the Aubin-Simon compactness result to the discrete setting see \cite{gal-12-comp, che-14-mac}.
Note however that the compactness result is somewhat stronger here than in the general Banach setting, in the sense that the estimate on the time translates is bounded by a power of the translation. 
This is possible thanks to the fact that the norms $\Vert \cdot \Vert_\ast$ and $\Vert \cdot \Vert^\ast$ are dual.

 \section{Regularity of the limit}

\begin{defi}[$B$-limit included sequence of spaces]\label{def-blinc}
Let $B$ be a Banach space, $(X_n)_\nnn$ be a sequence of Banach spaces included in $B$ and $X$ be a Banach space included in $B$.
We say that the sequence $(X_n)_{n \in \xN}$ is $B$-limit-included in $X$ if there exist $C \in \xR$ such that if $u$ is the limit in $B$ of a subsequence of a sequence $(u_n)_\nnn$ verifying $u_n\in X_n$ and $\norm{u_n}{X_n} \le 1$, then $u \in X$ and $\norm{u}{X} \le C$.
\end{defi}

\begin{theorem}[Regularity of the limit]\label{theo-regulimit} Let $1\le p < +\infty$ and $T>0$.
Let $B$ be a Banach space, $(X_n)_\nnn$ be a sequence of Banach spaces included in $B$ and $B$-limit-included in $X$ (where $X$ is a Banach space included in $B$) in the sense of Definition \ref{def-blinc}.

Let $T>0$ and, for $\nnn$, let $u_n \in L^p((0,T),X_n)$.
Let us assume that the sequence $(\norm{u_n}{L^p((0,T),X_n)})_\nnn$ is bounded and that $u_n \to u$ in $L^p((0,T),B)$ as $\nti$. 
Then $u \in L^p((0,T),X)$.
\end{theorem}

\begin{proof}
Since $u_n\to u$ in $L^p((0,T),B)$ as $\nti$, we can assume, up to subsequence, that $u_n \to u$ in $B$ a.e..

Then, since the the sequence $(X_n)_\nnn$ is $B$-limit-included in $X$, we obtain, with $C$ given by Definition \ref{def-blinc},
\[
\norm{u}{X} \le C \liminf_{\nti} \norm{u_n}{X_n} \textrm{~~a.e.}.
\]
Using now Fatou's lemma, we have
\[
\int_0^T \norm{u(t)}{X}^p dt \le C^p \int_0^T \liminf_{\nti} \norm{u_n(t)}{X_n}^p dt
\le C^p \liminf_{\nti} \int_0^T \norm{u_n(t)}{X_n}^p dt.
\]
Then, since $(\norm{u_n}{L^p((0,T),X_n)})_\nnn$ is bounded, we conclude that $u \in L^p((0,T),X)$. \end{proof}

%
% --------------------------------------------------------------------------------------------------------------------------------
%
\bibliographystyle{abbrv}
\bibliography{./rayt}
\end{document}